\documentclass[11pt,a4paper]{article}
\usepackage[pdftex]{graphicx}
\usepackage{amsmath,amssymb,amsfonts,amsthm}
\numberwithin{equation}{section}
\usepackage{indentfirst}
\usepackage{enumitem} 
\usepackage[margin=1.3in]{geometry}
\usepackage{fancyhdr}

\usepackage[colorlinks=true,linkcolor=magenta,citecolor=blue,urlcolor=cyan]{hyperref}
\usepackage{titlesec}
\usepackage{etoolbox}
\usepackage{graphicx}  
\usepackage{changepage}
\theoremstyle{plain}
\newtheorem{theorem}{Theorem}[section]
\newtheorem{lemma}[theorem]{Lemma}

\newtheorem{conjecture}{Conjecture}[section]

\makeatletter
\renewcommand{\maketitle}{
	\begin{center}
		{\Large\bfseries{\@title}\par}
		\vskip 1em
		{\normalsize
			\lineskip .5em
			\begin{tabular}[t]{c}
				\@author
			\end{tabular}\par}
		\vskip 1.5em
	\end{center}
}
\makeatother
\renewenvironment{abstract}{
	\begin{adjustwidth}{1.3cm}{1.3cm}
		\noindent{\large\bfseries{A{\scriptsize BSTRACT.}}}
	}{
	\end{adjustwidth}
}

\usepackage{color}
\usepackage{xcolor}
\usepackage[normalem]{ulem} 
\usepackage{soul}

\usepackage{graphicx} 
\usepackage{xstring}  

\usepackage{xstring}
\usepackage{graphicx} 

\usepackage{marvosym}

\begin{document}
	
	\title{Overcolored Partitions Restricted by Parity of the Parts }

	\author{M. P. Thejitha and S. N. Fathima}
	
	\maketitle
	
	\begin{abstract}
	Very recently, Thejitha, Sellers, and Fathima defined the function $a_{r,s}(n)$, which enumerates the number of multicolored partitions of $n$, wherein both even parts and odd parts may appear in one of $r$-colors and $s$-colors, respectively, for fixed $r,s\ge 1$. In this paper, we extend the concept to overpartitions.\\
		
		
	\end{abstract}
	
	\bigskip
	
	\vspace{0.5em}
	\section{Introduction}
			Let $n$ be a non-negative integer. Throughout, we use the following standard notations.
		\begin{align*}
			(a;q)_0:=1,\;\;(a;q)_n:=\prod_{j=0}^{n-1}(1-aq^j),	\;\; \text{and} \;\;(a;q)_\infty = \prod_{n=0}^\infty(1-aq^{n}).
		\end{align*}
		We denote by $p_k(n)$ the function which enumerates the partitions of $n$ into $k$ colors. The generating function for $p_k(n)$ is given by 
		\begin{align*}
			\sum_{n=0}^{\infty}p_k(n)q^n=\dfrac{1}{f_1^k},
		\end{align*}
		where $ f_m:= (q^m;q^m)_\infty$. Clearly, if $k=1$, then $p_1(n)$ is simply the unrestricted partition function. A partition of a positive integer $n$ is a finite sequence of non-increasing positive integers $\lambda_1, \lambda_2,..., \lambda_r$, called parts, such that $n=\sum_{i=1}^{r}\lambda_r$.\\
		\indent
		Agarwal and Andrews \cite{agar} were the first to apply color partitions to unrestricted partitions. Among many interesting perspectives of colored partitions that have been considered by several mathematicians, here we focus on the colored partitions, wherein both even and odd parts may appear in one of $r$-colors and $s$-colors, respectively, for fixed $r,s\ge 1$. These partitions are studied by Thejitha, Sellers, and Fathima \cite{tsf} and denoted by $a_{r,s}(n)$. The generating function for $a_{r,s}(n)$ is given by
		\begin{align}\label{gf0}
			\sum_{n=0}^{\infty}a_{r,s}(n)q^n=\dfrac{f_2^{s-r}}{f_1^s}.
		\end{align}
		The case $r=1$ are first studied by Hirschhorn and Sellers \cite{hs1}, while the case $s=1$ are first studied by  Amdeberhan, Sellers, and Singh \cite{ajsi}.\\
		\indent An overpartition of $n$ extends the unrestricted partitions by allowing the first occurrence of each part size to be overlined. The number of overpartitions of $n$ is denoted by $\bar{p}(n)$. For instance, the overpartition of $3$ are 
		\begin{align*}
			3, \bar{3}, 2+1, \bar{2}+1, 2+\bar{1}, \bar{2}+\bar{1}, 1+1+1, \bar{1}+1+1.
		\end{align*}
		The study of overpartitions dates back to MacMahon \cite{mac}, and was later revisited by Corteel and Lovejoy \cite{colo}, and has received a lot of attention. Corteel and Lovejoy \cite{colo} noted that the generating function for $\bar{p}(n)$ is given by
		\begin{align*}
			\sum_{n=0}^{\infty}\bar{p}(n)q^n=\dfrac{f_2}{f_1^2}.
		\end{align*}
		\indent A natural extension arises by considering overpartitions that are colored. As in \cite{ajsi}, Amdeberhan, Sellers, and Singh extended their work ($s=1$ case of \eqref{gf0}) to overpartitions. The  number of partitions of $n$ wherein each even part may appear in $r\ge1$  different colors and the first occurrence of parts may be over-lined, denoted by $\bar{a}_r^*(n)$, has the generating function
		\begin{align}\label{1}
			\sum_{n=0}^{\infty}\bar{a}_r^*(n)q^n=\dfrac{f_4^{r-1}}{f_1^2f_2^{2r-3}}.
		\end{align}
		\noindent Recently, Paksok and Saikia \cite{pak} proved new Ramanujan-type congruence modulo $8$ and $16$ for $\bar{a}_{2m+1}^*(n)$, $\bar{a}_{2m+2}^*(n)$, and $\bar{a}_{8m+3}^*(n)$.
		Later, Das et. al \cite{das} have studied the divisibility properties of $\bar{a}_r^*(n)$ for different values of $r$. In particular, they proved the following generalizations of Seller's results {{\cite[Th. 2.5 and Th. 2.7]{ovrcu}}}.
		\begin{theorem}[{{\cite[Theorem~1.4]{das}}}]\label{dasm4}
			For all $n\ge 1$,
			\[
			\bar{a}_{r}^*(n)\equiv
			\begin{cases}
				2\pmod 4, & \text{if } n=k^2, k\in \mathbb{Z},\\[4pt]
				2(r+1)\pmod 4, & \text{if } n=2k^2, k\in \mathbb{Z},\\[4pt]
				0 \pmod 4, & \text{otherwise }.
			\end{cases}
			\]
		\end{theorem}
		\begin{theorem}[{{\cite[Theorem~1.6]{das}}}]\label{dasm8}
			For all $n\ge 1$,
			\[
			\bar{a}_{r}^*(n)\equiv
			\begin{cases}
				2\pmod 8, & \text{if } n=k^2, k\in\mathbb{Z},\\[4pt]
				6(r+1)\pmod 8, & \text{if } n=2(2k)^2, k\in\mathbb{Z},\\[4pt]
				2(r+1)\pmod 8, & \text{if } n=2(2k-1)^2, k\in\mathbb{Z},\\[4pt]
				2(r+1)(r+2)\pmod 8,  & \text{if } n=4k^2, k\in\mathbb{Z},\\[4pt]
				4(r+1)\pmod 8,  & \text{if } n=k^2+2\ell^2, k\in\mathbb{Z},\\[4pt]
				0\pmod 8,  & \text{otherwise }.\\[4pt]
			\end{cases}
			\]
		\end{theorem}

		It is easy to note that number of partitions of $n$ wherein each odd part may appear in $s\ge1$  different colors and the first occurrence of parts may be over-lined, denoted by $\bar{a}_s(n)$, has the generating function
		\begin{align}\label{2}
			\sum_{n=0}^{\infty}\bar{a}_s(n)q^n=\dfrac{f_2^{3s-2}}{f_1^{2s}f_4^{s-1}}\; .
		\end{align}
		Our goal in this work is to consider $\bar{a}_{r,s}(n)$, the number of partitions of $n$ wherein both even and odd parts may appear in one of $r$-colors and $s$-colors, respectively, for fixed $r,s\ge 1$ and the first occurrence of parts may be over-lined. We work out to obtain the generating function for $\bar{a}_{r,s}(n)$ to be
		\begin{align}\label{gf}
			\bar{A}_{r,s}(q):=	\sum_{n=0}^{\infty}\bar{a}_{r,s}(n)q^n=\dfrac{f_2^{3s-2r}}{f_1^{2s}f_4^{s-r}}\; .
		\end{align}
		It is interesting to note, the generating function \eqref{gf} unifies \eqref{1} and \eqref{2}. In this paper, we prove several infinite families of congruences modulo prime and modulo powers of $2$ for $\bar{a}_{r,s}(n)$ which generalizes Theorem \ref{dasm4} and \ref{dasm8}. The main results of this paper pertains to several Ramanujan-type congruences modulo powers of $2$ and modulo prime $p\ge3$, satisfied by $\bar{a}_{r,s}$. 
		
		The following Theorem \ref{tmod4} and \ref{tmod8} are generalization of Theorem \ref{dasm4} and \ref{dasm8}, respectively. 
		\begin{theorem}\label{tmod4}
			For all $r,s,n\ge 1$, and some integer $k\ge1$,
			\[
			\bar{a}_{r,s}(n)\equiv
			\begin{cases}
				2s\pmod 4, & \text{if } n=k^2,\\[4pt]
				2(r+s)\pmod 4, & \text{if } n=2k^2,\\[4pt]
				0 \pmod 4, & \text{otherwise }.
			\end{cases}
			\]
		\end{theorem}
		\begin{theorem}\label{tmod8}
			For all $r,s,n\ge 1$, and some integers $k,\ell\ge1$,
		\[
		\bar{a}_{r,s}(n)\equiv
		\begin{cases}
			2s\pmod 8, & \text{if } n=(2k-1)^2,\\[4pt]
			2(s+(r+s)(r+s+1))\pmod 8, & \text{if } n=(2k)^2,\\[4pt]
			2(r+s^2)\pmod 8, & \text{if } n=2(2k-1)^2,\\[4pt]
			2(3r+s^2+2s)\pmod 8,  & \text{if } n=2(2k)^2,\\[4pt]
			4(r+s)\pmod 8,  & \text{if } n=k^2+2\ell^2,\\[4pt]
			0\pmod 8,  & \text{otherwise }.\\[4pt]
		\end{cases}
		\]
		\end{theorem}
		In the next two theorems we prove	infinite families of congruence modulo powers of $2$.
		\begin{theorem}\label{t1}
			For all $n\ge1$, $r,s,i,j\ge 1$, and for $k\in \mathbb{Z}$, we have
			\begin{align}
				\bar{a}_{r,s}(n)&\equiv 0 \pmod 2 \label{e1.151}\\
				\bar{a}_{r,s}(4n+2)&\equiv 0\pmod 4, \quad\text{when }  r+s=2k \label{e1.52}\\
				\bar{a}_{2j,2i+1}(4n+3)&\equiv 0 \pmod 4 \label{e1.53}\\
				\bar{a}_{r,2i}(3n+1)&\equiv 0 \pmod 4 \label{e3n1m4}\\
				\bar{a}_{r,s}(3n+2)&\equiv 0 \pmod 4, \quad \text{when } r+s=2k\label{e3n2m4}\\
				\bar{a}_{r,s}(9n+3)&\equiv 0 \pmod 4 \label{e9n3m4}\\
				\bar{a}_{r,s}(9n+6)&\equiv 0 \pmod 4 \label{e9n6m4}\\
				\bar{a}_{2j+1, s}(9n+3)&\equiv 0 \pmod 8 \label{e9n3m8}\\
				\bar{a}_{2j+1, s}(9n+6)&\equiv 0 \pmod 8 \label{e9n6m8}.
			\end{align} 
		\end{theorem}
		
		\begin{theorem}\label{t1.8}
			For all $n\ge 0$ and $r,k,i,j\ge 1$, we have
			\begin{align}
				\bar{a}_{r,2^k}(2n+1)&\equiv0\pmod {2^{k+1}} \label{e1.61}\\
				\bar{a}_{r,2^ki+2^{k-1}}(2n+1)&\equiv0\pmod {2^{k}} \label{e1.62} \\
				\bar{a}_{2^{k+1}j+2^k-1, 2^ki+1}(4n+2)&\equiv 0\pmod {2^{k+1}}\label{e1.63}\\
				\bar{a}_{2^{k+1}j+2^k, 2^ki}(4n+2)&\equiv 0\pmod {2^{k+1}}\label{e1.64}\\
				\bar{a}_{2^{k+1}j+2^k-1, 2^ki+1}(4n+3)&\equiv 0\pmod {2^{k+2}}\label{e1.65}\\
				\bar{a}_{2^{k+1}j+2^k, 2^ki}(4n+3)&\equiv 0\pmod {2^{k+2}}\label{e1.66}\\
				\bar{a}_{2^{k+1}j+2^k, 2^ki}(8n+4)&\equiv 0 \pmod {2^{k+1}}\label{e1.68}\\
				\bar{a}_{2^{k+1}j+2^k-1, 2^ki+1}(8n+5)&\equiv 0 \pmod {2^{k+2}}.\label{e1.69}
			\end{align}
		\end{theorem}
		\noindent More generally, we present infinite families of congruences modulo prime $p\ge3$.
		\begin{theorem}\label{tmod3}
			For all $n\ge0$ and $k\ge j\ge0$, we have
			\begin{align}
				\bar{a}_{3(k-j)+3, 3k+6}(3n+1)\equiv 0 \pmod 3 \label{e1.20}\\
				\bar{a}_{3(k-j)+3, 3k+6}(3n+2)\equiv 0 \pmod 3 \label{e1.21}\\
				\bar{a}_{3(k-j)+5, 3k+1}(3n+1)\equiv 0 \pmod 3 \label{e1.22}.
			\end{align}
		\end{theorem}
		
		\begin{theorem}\label{tmodp}
			For all $n\ge0$, $k\ge j\ge0$, and prime $p\ge 3$, with  $1\le r\le p-1$, such that 
			\begin{enumerate}
				\item $r$ is a quadratic non-residue modulo $p$,
				\begin{align}
					\bar{a}_{p(k-j)+(p-1), pk+(p+1)} (pn+r)\equiv 0\pmod {p}, \label{e1.23}
				\end{align}
				\item $2^{-1}r$ is a quadratic non-residue modulo $p$, 
				\begin{align}
					\bar{a}_{p(k-j)+(p-1), pk+p} (pn+r)\equiv 0\pmod {p}.\label{e1.24}
				\end{align}
			\end{enumerate}
		\end{theorem}
		We organize the rest of the paper as follows. In Section \ref{s2}, we present the definition of theta function $\phi(q)$ and dissection formula for $\dfrac{1}{f_1^2}$ critical to prove our main results. In Section \ref{s3}-\ref{s5} we demonstrate Theorem \ref{tmod4}-\ref{tmodp} using elementary $q$-operations which may motivate further avenues of research.
		
		\section{Preliminaries}\label{s2}
		In this section, we collect lemmas that are essential in the proofs of our main results. Following {{\cite[p. 34]{rnb}}} and {{\cite[Eq. 1.2.1]{spi}}}, we are primarily concerned with Ramanujan's theta function $f(a,b)$ defined as
		\begin{align}
			f(a,b):=\sum_{n=-\infty}^{\infty}a^{n(n+1)/2}b^{n(n-1)/2},\;\; |ab|<1 \label{e2.1}. 
		\end{align} 
		Ramanujan \cite{rnb} also rediscovered the classical Jacobi's triple-product identity 
		\begin{align*}
			f(a,b)=(-a;ab)_\infty(-b;ab)_\infty(ab;ab)_\infty.
		\end{align*}
		It is worth noting here that (see {{\cite[Entry 22, p. 36-37]{rnb}}})
		\begin{align}
			\phi(q):&=f(q,q)=1+2\sum_{n=1}^{\infty}q^{n^2}=\dfrac{f_2^5}{f_1^2f_4^2}\label{phi}\\
			\phi(-q)&=\dfrac{f_1^2}{f_2}\label{phim}.
		\end{align}
		We record the following congruence which can be easily obtained by employing binomial theorem.
		For any positive integer $k$ and $m$, and prime $p$, we have
		\begin{align*}
			f_m^{p^k} \equiv f_{mp}^{p^{k-1}} \pmod{p^k}.
		\end{align*}
		We now recollect the key dissection of $\dfrac{1}{f_1^2}$ from Hirschhorn \cite{poq}.
		\begin{lemma}[{{\cite[Equation 1.9.4]{poq}}}]\label{ld}
			We have
			\begin{align}
				\dfrac{1}{f_1^2}=\dfrac{f_8^5}{f_2^5f_{16}^2}+2q\dfrac{f_4^2f_{16}^2}{f_2^5f_8}.
			\end{align}
		\end{lemma}
		\noindent To prove main theorems, we also require the following precursory result proved by Thejitha, Sellers, and Fathima \cite{tsf}.
		\begin{lemma}[{{\cite[Lemma~2.6]{tsf}}}]\label{lmodp}
			For all prime $p$, and positive integers $\alpha_i, \beta_i, \gamma_i, \delta_i, a_i, b_i, \text{and} \;\lambda$, define \\
			\begin{align*}
				\displaystyle{\sum_{n=0}^{\infty}A(n)q^n:=\dfrac{\prod_{i=1}^{j}f_{\alpha_i}^{\gamma_i}}{\prod_{i=1}^{k}f_{\beta_i}^{\delta_i}}}
				\quad	\text{and} \quad \displaystyle{\sum_{n=0}^{\infty}B(n)q^n:=\dfrac{\prod_{i=1}^{l}f_{\alpha_i}^{a_ip^\lambda+\gamma_i}}{\prod_{i=1}^{m}f_{\beta_i}^{b_ip^\lambda+\delta_i}}}. \\
			\end{align*}		
			For all $n\geq 0$ and $1\le C\le p^\lambda -1$, if $A(p^\lambda n+C)\equiv 0\pmod p$, then $B(p^\lambda n+C)\equiv 0\pmod p$.
		\end{lemma}

		\section{Proof of Theorem \ref{tmod4}-\ref{tmod8}} \label{s3}
		In this section, we first prove the following lemma which is instrumental in proving our main results.
		\begin{lemma}\label{cai}
			For all $r,s\ge 1$, we have 
			\begin{align*}
				\bar{A}_{r,s}(q)=\phi(q)^s\prod_{i=1}^{\infty}\phi(q^{2^i})^{(r+s)\cdot2^{i-1}}.
			\end{align*}
		\end{lemma}
		\begin{proof} Thanks to \eqref{gf}, we have
			\begin{align*}
				\bar{A}_{r,s}(q)&=\dfrac{f_2^{3s-2r}}{f_1^{2s}f_4^{s-r}}\\
				&=\dfrac{f_2^{5s}}{f_1^{2s}f_4^{2s}}\cdot \dfrac{f_4^{5r-5s}}{f_2^{2r-2s}f_8^{2r-2s}}\cdot \dfrac{f_4^{6s-4r}}{f_2^{4s}f_8^{2s-2r}}.
			\end{align*}
			We apply \eqref{phi} so that
			\begin{align*}
				\bar{A}_{r,s}(q)=\phi(q)^s\phi(q^2)^{r-s}\bar{A}_{r,s}(q^2)^2.
			\end{align*}
			Now, in the above equation we proceed by iteration to complete the proof of Lemma \ref{cai}.
		\end{proof}
		At last, we have enough to prove our main results.
		\begin{proof}[\textbf{Proof of Theorem \ref{tmod4}}]
			Using Lemma \ref{cai} and employing the fact, for all $i\ge 2$, $\phi(q^{2^i})^{(r+s)\cdot2^{i-1}}\equiv 1 \pmod 4$, we obtain
			\begin{align*}
				\sum_{n=0}^{\infty}\bar{a}_{r,s}(n)q^n&=\phi(q)^s\prod_{i=1}^{\infty}\phi(q^{2^i})^{(r+s)\cdot2^{i-1}}\\
				&\equiv \phi(q)^s\phi(q^2)^{r+s}\pmod 4.
			\end{align*}
			Thanks to \eqref{phi}, we have
			\begin{align*}
				\sum_{n=0}^{\infty}\bar{a}_{r,s}(n)q^n&\equiv\displaystyle\left(1+2\sum_{n=1}^{\infty}q^{n^2}\right)^s\displaystyle\left(1+2\sum_{n=1}^{\infty}q^{2n^2}\right)^{r+s}\pmod 4\\
				&=\displaystyle\left(\sum_{j=0}^{s}\binom{s}{j}2^j\displaystyle\left(\sum_{n=1}^{\infty}q^{n^2}\right)^j\right)\displaystyle\left(\sum_{k=0}^{r+s}\binom{r+s}{k}2^k\displaystyle\left(\sum_{n=1}^{\infty}q^{2n^2}\right)^k\right)\\
				&\equiv\displaystyle\left(1+2s\sum_{n=1}^{\infty}q^{n^2}\right)\displaystyle\left(1+2(r+s)\sum_{n=1}^{\infty}q^{2n^2}\right)\pmod 4\\
				&\equiv1+2s\sum_{n=1}^{\infty}q^{n^2}+2(r+s)\sum_{n=1}^{\infty}q^{2n^2}\pmod 4.
			\end{align*}
			This completes the proof of Theorem \ref{tmod4}.
		\end{proof}

		\begin{proof}[\textbf{Proof of Theorem \ref{tmod8}}]
			Using Lemma \ref{cai} and employing the fact, for all $i\ge 3$, $\phi(q^{2^i})^{(r+s)\cdot2^{i-1}}\equiv 1 \pmod 8$, we obtain
			\begin{align*}
				\sum_{n=0}^{\infty}\bar{a}_{r,s}(n)q^n&\equiv \phi(q)^s\phi(q^2)^{r+s}\phi(q^4)^{2(r+s)}\pmod 8.
			\end{align*}
			Again, thanks to \eqref{phi}, we have
			\begin{align*}
				\sum_{n=0}^{\infty}\bar{a}_{r,s}(n)q^n	&\equiv\displaystyle\left(1+2\sum_{n=1}^{\infty}q^{n^2}\right)^s\displaystyle\left(1+2\sum_{n=1}^{\infty}q^{2n^2}\right)^{r+s}\displaystyle\left(1+2\sum_{n=1}^{\infty}q^{4n^2}\right)^{2(r+s)}\pmod 8\\
				&=\displaystyle\left(\sum_{j=0}^{s}\binom{s}{j}2^j\displaystyle\left(\sum_{n=1}^{\infty}q^{n^2}\right)^j\right)\displaystyle\left(\sum_{k=0}^{r+s}\binom{r+s}{k}2^k\displaystyle\left(\sum_{n=1}^{\infty}q^{2n^2}\right)^k\right)\\
				&\quad\displaystyle\left(\sum_{l=0}^{2(r+s)}\binom{2(r+s)}{l}2^l\displaystyle\left(\sum_{n=1}^{\infty}q^{4n^2}\right)^l\right)\\
				&\equiv\displaystyle\left(1+2s\sum_{n=1}^{\infty}q^{n^2}+2s(s-1)\sum_{n=1}^{\infty}q^{2n^2}\right)\\&\quad
				\displaystyle\left(1+2(r+s)\sum_{n=1}^{\infty}q^{2n^2}+2(r+s-1)(r+s)\sum_{n=1}^{\infty}q^{4n^2}\right)\\
				&\quad\displaystyle\left(1+4(r+s)\sum_{n=1}^{\infty}q^{4n^2}+4(r+s)(2r+2s-1)\sum_{n=1}^{\infty}q^{8n^2}\right)\pmod 8\\
				&\equiv1+2s\sum_{n=1}^{\infty}q^{(2n-1)^2}+2(s+(r+s)(r+s+1))\sum_{n=1}^{\infty}q^{(2n)^2}\\
				&\quad+2(r+s^2)\sum_{n=1}^{\infty}q^{2(2n-1)^2}+2(3r+2s+s^2)\sum_{n=1}^{\infty}q^{2(2n)^2}\\
				&\quad+4s(r+s)\sum_{m,n=1}^{\infty}q^{n^2+2m^2}\pmod 8.
			\end{align*}
			This completes proof of Theorem \ref{tmod8}.
		\end{proof}

		\section{Congruences modulo powers of $2$}
		In this section, we establish Theorem \ref{t1} and \ref{t1.8}.
		\begin{proof}[\textbf{Proof of Theorem \ref{t1}}]
			The proof of \eqref{e1.151} follows immediately from Lemma \ref{cai}.\\
			\indent	Next, from \eqref{gf} and Lemma \ref{ld}, we have
			\begin{align}
				\sum_{n=0}^{\infty}\bar{a}_{r,s}(n)q^n
				&=\dfrac{f_2^{3s-2r}}{f_4^{s-r}}\displaystyle\left(\dfrac{f_8^5}{f_2^5f_{16}^2}+2q\dfrac{f_4^2f_{16}^2}{f_2^5f_8}\right)^{s} \nonumber\\
				&=\dfrac{f_2^{3s-2r}}{f_4^{s-r}}\sum_{m=0}^{s}\binom{s}{m}2^m q^m \displaystyle\left(\dfrac{f_4^2f_{16}^2}{f_2^5f_8}\right)^m\displaystyle\left(\dfrac{f_8^5}{f_2^5f_{16}^2}\right)^{s-m} \nonumber\\
				&=\dfrac{f_8^{5s}}{f_2^{2(s+r)}f_4^{s-r}f_{16}^{2s}}\sum_{m=0}^{s}\binom{s}{m}2^mq^m\dfrac{f_4^{2m}f_{16}^{4m}}{f_8^{6m}}.\label{en}
			\end{align}
			Extracting the terms involving $q^{2n}$ and $q^{2n+1}$ in the expression \eqref{en}, yields
			\begin{align*}
				\sum_{n=0}^{\infty}\bar{a}_{r,s}(2n)q^n&=\dfrac{f_4^{5s}}{f_1^{2(s+r)}f_2^{s-r}f_{8}^{2s}}\sum_{m=0}^{s}\binom{s}{2m}2^{2m}q^m\dfrac{f_2^{4m}f_{8}^{8m}}{f_4^{12m}},\\
				\sum_{n=0}^{\infty}\bar{a}_{r,s}(2n+1)q^n&=\dfrac{f_4^{5s-6}f_2^{2-s+r}}{f_1^{2(s+r)}f_{8}^{2s-4}}\sum_{m=0}^{s}\binom{s}{2m+1}2^{2m+1}q^m\dfrac{f_2^{4m}f_{8}^{8m}}{f_4^{12m}},
			\end{align*}
			respectively.\\
			Now, applying Lemma \ref{ld} again in the equations above, we obtain 
			\begin{align}
				\sum_{n=0}^{\infty}\bar{a}_{r,s}(2n)q^n&=\dfrac{f_4^{5s}f_8^{3s+5r}}{f_2^{2(3s+2r)}f_{16}^{2(s+r)}}\displaystyle\left(\sum_{m=0}^{s}\binom{s}{2m}2^{2m}q^m\dfrac{f_2^{4m}f_{8}^{8m}}{f_4^{12m}}\right)\nonumber\\
				&\quad \displaystyle\left(\sum_{l=0}^{s+r}\binom{s+r}{l}2^lq^l \dfrac{f_4^{2l}f_{16}^{4l}}{f_8^{6l}}\right),\label{e2n}\\
				\sum_{n=0}^{\infty}\bar{a}_{r,s}(2n+1)q^n&=2\dfrac{f_4^{5s-6}f_8^{3s+5r+4}}{f_2^{2(3s+2r-1)}f_{16}^{2(s+r)}}\displaystyle\left(\sum_{m=0}^{s}\binom{s}{2m+1}2^{2m}q^m\dfrac{f_2^{4m}f_{8}^{8m}}{f_4^{12m}}\right)\nonumber\\
				&\quad \displaystyle\left(\sum_{l=0}^{s+r}\binom{s+r}{l}2^lq^l \dfrac{f_4^{2l}f_{16}^{4l}}{f_8^{6l}}\right),\label{e21n}
			\end{align}
			respectively.\\
			\indent Now, to prove \eqref{e1.52} of Theorem \ref{t1} we simplify the equation \eqref{e2n} upon using the fact, for $s, m, l,k\ge 1$,
			and $r+s=2k$,
			\begin{align*}
				2^{2m}\binom{s}{2m} \equiv 0\pmod 4 \quad \text{and} \quad 2^l\binom{s+r}{l}\equiv 0 \pmod 4,
			\end{align*}
			to obtain				\begin{align*}
				\sum_{n=0}^{\infty}\bar{a}_{r,s}(2n)q^n&\equiv\dfrac{f_4^{5s}f_8^{3s+5r}}{f_2^{2(3s+2r)}f_{16}^{2(s+r)}}\pmod {4}.
			\end{align*}
			Extracting the terms involving $q^{2n+1}$ in the above equation, we arrive at the desired result.\\
			\indent Next, setting $(r,s)=(2j,2i+1)$ in \eqref{e21n}, we obtain
			\begin{align*}
				\sum_{n=0}^{\infty}\bar{a}_{2j,2i+1}(2n+1)q^n&=2\dfrac{f_4^{10i-1}f_8^{6i+10j+7}}{f_2^{4(3i+2j+1)}f_{16}^{2(2i+1+2j)}}\displaystyle\left(\sum_{m=0}^{i}\binom{2i+1}{2m+1}2^{2m}q^m\dfrac{f_2^{4m}f_{8}^{8m}}{f_4^{12m}}\right)\\
				&\quad \displaystyle\left(\sum_{l=0}^{2j+2i+1}\binom{2i+2j+1}{l}2^lq^l \dfrac{f_4^{2l}f_{16}^{4l}}{f_8^{6l}}\right),
			\end{align*}
			which upon using the fact, for $m, l\ge 1$,
			\begin{align*}
				2^{2m}\binom{2i+1}{2m+1}\equiv 0\pmod {2} \quad \text{and} \quad 2^l\binom{2i+2j+1}{l}\equiv 0\pmod {2},
			\end{align*}
			yields
			\begin{align*}
				\sum_{n=0}^{\infty}\bar{a}_{2j,2i+1}(2n+1)q^n&=2\dfrac{f_4^{10i-1}f_8^{6i+10j+7}}{f_2^{4(3i+2j+1)}f_{16}^{2(2i+1+2j)}}\pmod 4.
			\end{align*}
			Extracting the terms involving $q^{2n+1}$ in above expression, evidently completes the proof of \eqref{e1.53}.\\
			\indent For the poofs of congruences \eqref{e3n1m4}-\eqref{e9n6m4}, we shift our attention to Theorem \ref{tmod4}.\\
			\indent We first note that for $s=2i$ in Theorem \ref{tmod4}, if $3n+1=k^2$, then we have the divisibility by $4$. Also, if $3n+1=2k^2$, we have $2\equiv (2k)^2\pmod 3$, which is a contradiction, since $2$ is a quadratic non-residue modulo $3$. This settles \eqref{e3n1m4}.\\
			\indent	Again, for $3n+2=k^2$, we have $2\equiv k^2\pmod 3$, which is not true, since $2$ is a quadratic non-residue modulo $3$. Now, it is easy to observe for $r+s=2\ell$, in Theorem \ref{tmod4}, if $3n+2=2k^2$, then we have the divisibility by $4$. This settles \eqref{e3n2m4}.\\
			\indent Now, for $9n+3=k^2$, we note $9n+3=3(3n+1)$ is not divisible by $9$. Thus, $9n+3$ is not a square. Similarly, for $9n+3=2k^2$, we have $3(6n+2)=(2k)^2$. Since $3(6n+2)$ is not divisible by $9$, implies $18n+6$ is not a square. Therefore, \eqref{e9n3m4} follows from Theorem \ref{tmod4}.\\
			\indent Lastly, for $9n+6=k^2$, we observe $9n+6=3(3n+4)$ is not divisible by $9$. Thus, $9n+6$ is not a square. Similarly, for $9n+6=2k^2$, we have $3(6n+4)=(2k)^2$. Since $3(6n+2)$ is not divisible by $9$, implies $18n+12$ is not square. Therefore, \eqref{e9n6m4} follows from Theorem \ref{tmod4}.\\
			\indent In order to prove the congruences \eqref{e9n3m8} and \eqref{e9n6m8}, we use Theorem \ref{tmod8}.\\
			\indent It is easy to observe $9n+3=k^2$ and $9n+3=4k^2=(2k)^2$ are not true, since $9n+3$ is not a square. Similarly, for  $9n+3=2(2k)^2$ and $9n+3=2(2k-1)^2$, we have $3(6n+2)=(4k)^2$ and $3(6n+2)=(2(2k-1))^2$, respectively which is again not true. Therefore, for any integer $s$ and odd $r$ in Theorem \ref{tmod8}, if $9n+3=k^2+2l^2$, then we have the divisibility by $8$. This settles \eqref{e9n3m8}.\\
			\indent It is easy to see, $9n+6=k^2$ and $9n+6=4k^2$ are not true. Similarly, for $9n+6=2(2k)^2$ and  $9n+6=2(2k-1)^2$, we have $3(6n+4)=(4k)^2$ and $3(6n+4)=(2(2k-1))^2$, respectively, which is again not true. Therefore, for  any integer $s$ and odd $r$ in Theorem \ref{tmod8}, if $9n+6=k^2+2l^2$, then we have the divisibility by $8$. This settles \eqref{e9n6m8}.
		\end{proof}
		
		\begin{proof}[\textbf{Proof of Theorem \ref{t1.8}}]
			Setting $(r,s)=(r,2^k)$ and $(r,2^ki+2^{k-1})$ in \eqref{en}, we obtain the equations
			\begin{align*}
				\sum_{n=0}^{\infty}\bar{a}_{r,2^k}(n)q^n&=\dfrac{f_8^{5\cdot2^k}}{f_2^{2(2^k+r)}f_4^{2^k-r}f_{16}^{2^{k+1}}}\sum_{m=0}^{2^k}\binom{2^k}{m}2^mq^m\dfrac{f_4^{2m}f_{16}^{4m}}{f_8^{6m}},\\
				\sum_{n=0}^{\infty}	\bar{a}_{r,2^ki+2^{k-1}}(n)q^n&=\dfrac{f_8^{5(2^ki+2^{k-1})}}{f_2^{2(2^ki+2^{k-1}+r)}f_4^{(2^ki+2^{k-1})-r}f_{16}^{2(2^ki+2^{k-1})}}\cdot\\ &\quad\displaystyle\left(\sum_{j=0}^{2^ki+2^{k-1}}\binom{2^ki+2^{k-1}}{j}2^jq^j\displaystyle\left(\dfrac{f_4^{2j}f_{16}^{4j}}{f_8^{6j}}\right)\right),
			\end{align*}
			respectively, which upon using the fact, for $m\ge1$,
			\begin{align*}
				2^m \binom{2^k}{m}\equiv 0 \pmod {2^{k+1}} \quad\text{and} \quad	2^j \binom{2^ki+2^{k-1}}{j}\equiv 0 \pmod {2^{k}},
			\end{align*}
			yields
			\begin{align*}
				\sum_{n=0}^{\infty}\bar{a}_{r,2^k}(n)q^n&\equiv \dfrac{f_8^{5\cdot2^k}}{f_2^{2(2^k+r)}f_4^{2^k-r}f_{16}^{2^{k+1}}} \pmod {2^{k+1}},\\
				\sum_{n=0}^{\infty}	\bar{a}_{r,2^ki+2^{k-1}}(n)q^n&\equiv\dfrac{f_8^{5(2^ki+2^{k-1})}}{f_2^{2(2^ki+2^{k-1}+r)}f_4^{(2^ki+2^{k-1})-r}f_{16}^{2(2^ki+2^{k-1})}} \pmod {2^{k}},
			\end{align*}
			respectively.
			Extracting the terms involving $q^{2n+1}$ in the above two expressions, we arrive at the desired results \eqref{e1.61} and \eqref{e1.62}, respectively. \\
			\indent	Setting $(r,s)=(2^{k+1}j+2^k-1, 2^ki+1)$ and $(2^{k+1}j+2^k, 2^ki)$ in \eqref{e2n}, we obtain the equations
			\begin{align*}
				\sum_{n=0}^{\infty}\bar{a}_{2^{k+1}j+2^k-1, 2^ki+1}(2n)q^n&=\dfrac{f_4^{5(2^ki+1)}f_8^{3(2^ki)+5(2^{k+1}i+2^k)-2}}{f_2^{2(3(2^ki)+2(2^{k+1}j+2^k)+1)}f_{16}^{2(2^{k+1}j+2^k+2^ki)}}\\
				&\quad\displaystyle\left(\sum_{m=0}^{2^ki+1}\binom{2^ki+1}{2m}2^{2m}q^m\dfrac{f_2^{4m}f_{8}^{8m}}{f_4^{12m}}\right)\nonumber\\
				&\quad \displaystyle\left(\sum_{l=0}^{2^{k+1}j+2^k+2^ki}\binom{2^{k+1}j+2^k+2^ki}{l}2^lq^l \dfrac{f_4^{2l}f_{16}^{4l}}{f_8^{6l}}\right),
			\end{align*}
			\begin{align*}
				\sum_{n=0}^{\infty}\bar{a}_{2^{k+1}j+2^k, 2^ki}(2n)q^n&=\dfrac{f_4^{5(2^{k}i)}f_8^{3(2^ki)+5(2^{k+1}j+2^k)}}{f_2^{2(3(2^ki)+2(2^{k+1}j+2^k))}f_{16}^{2(2^{k+1}j+2^k+2^ki)}}\\
				&\quad\displaystyle\left(\sum_{m=0}^{2^ki}\binom{2^ki}{2m}2^{2m}q^m\dfrac{f_2^{4m}f_{8}^{8m}}{f_4^{12m}}\right)\nonumber\\
				&\quad \displaystyle\left(\sum_{l=0}^{2^{k+1}j+2^k+2^ki}\binom{2^{k+1}j+2^k+2^ki}{l}2^lq^l \dfrac{f_4^{2l}f_{16}^{4l}}{f_8^{6l}}\right),
			\end{align*}
			respectively. For $m\ge 1$, invoking
			\begin{align*}
				2^{2m}\binom{2^ki}{2m}&\equiv 0\pmod {2^{k+1}}, \quad
				2^{2m}\binom{2^ki+1}{2m}\equiv 0\pmod {2^{k+1}},
				\text{ and }\\ 2^l\binom{2^{k+1}j+2^k+2^ki}{l}&\equiv 0\pmod {2^{k+1}},
			\end{align*}
			in the above equations and extracting the terms involving $q^{2n+1}$, we complete the proof of \eqref{e1.63} and \eqref{e1.64}.\\
			\indent Setting $(r,s)=(2^{k+1}j+2^k-1, 2^ki+1)$ and $(r,s)=(2^{k+1}j+2^k, 2^ki)$ in \eqref{e21n}, we obtain the equations
			\begin{align*}
				\sum_{n=0}^{\infty}\bar{a}_{2^{k+1}j+2^k-1, 2^ki+1}(2n+1)q^n&=2\dfrac{f_4^{5(2^ki)-1}f_8^{3(2^ki)+5(2^{k+1}j+2^k)+2}}{f_2^{2(3(2^ki)+2(2^{k+1}j+2^k))}f_{16}^{2(2^ki+2^{k+1}j+2^k)}}\\
				&\quad\displaystyle\left(\sum_{m=0}^{2^ki+1}\binom{2^ki+1}{2m+1}2^{2m}q^m\dfrac{f_2^{4m}f_{8}^{8m}}{f_4^{12m}}\right)\\
				&\quad \displaystyle\left(\sum_{l=0}^{2^{k+1}j+2^k+2^ki}\binom{2^{k+1}j+2^k+2^ki}{l}2^lq^l \dfrac{f_4^{2l}f_{16}^{4l}}{f_8^{6l}}\right),\\
				\sum_{n=0}^{\infty}\bar{a}_{2^{k+1}j+2^k, 2^ki}(2n+1)q^n&=2\dfrac{f_4^{5(2^ki)-6}f_8^{3(2^ki)+5(2^{k+1}j+2^k)+4}}{f_2^{2(3(2^ki)+2(2^{k+1}j+2^k)-1)}f_{16}^{2(2^ki+2^{k+1}j+2^k)}}\\&\quad\displaystyle\left(\sum_{m=0}^{2^ki}\binom{2^ki}{2m+1}2^{2m}q^m\dfrac{f_2^{4m}f_{8}^{8m}}{f_4^{12m}}\right)\\
				&\quad \displaystyle\left(\sum_{l=0}^{2^{k+1}j+2^k+2^ki}\binom{2^{k+1}j+2^k+2^ki}{l}2^lq^l \dfrac{f_4^{2l}f_{16}^{4l}}{f_8^{6l}}\right),
			\end{align*}	
			respectively. For $m, l\ge 1$, applying
			\begin{align*}
				2^{2m}\binom{2^ki}{2m+1}&\equiv 0\pmod {2^{k+1}},
				2^{2m}\binom{2^ki+1}{2m+1}\equiv 0\pmod {2^{k+1}}, \text{ and }\\ 2^l\binom{2^{k+1}j+2^k+2^ki}{l}&\equiv 0\pmod {2^{k+1}},
			\end{align*}
			in the above equations and then extracting the terms involving $q^{2n+1}$, we obtain \eqref{e1.65} and \eqref{e1.66}.\\
			\indent 	Now, setting $(r,s)=(2^{k+1}j+2^k, 2^ki)$ in \eqref{e2n}, we obtain
			\begin{align*}
				\sum_{n=0}^{\infty}\bar{a}_{2^{k+1}j+2^k, 2^ki}(2n)q^n&=\dfrac{f_4^{5(2^ki)}f_8^{3(2^ki)+5(2^{k+1}j+2^k)}}{f_2^{2(3(2^ki)+2(2^{k+1}j+2^k))}f_{16}^{2(2^ki+2^{k+1}j+2^k)}}\\
				&\quad\displaystyle\left(\sum_{m=0}^{2^ki}\binom{2^ki}{2m}2^{2m}q^m\dfrac{f_2^{4m}f_{8}^{8m}}{f_4^{12m}}\right)\nonumber\\
				&\quad \displaystyle\left(\sum_{l=0}^{2^{k+1}j+2^k+2^ki}\binom{2^{k+1}j+2^k+2^ki}{l}2^lq^l \dfrac{f_4^{2l}f_{16}^{4l}}{f_8^{6l}}\right).
			\end{align*}
			Using the fact, for $m,l\ge 1$,
			\begin{align*}
				2^{2m}\binom{2^ki}{2m}\equiv 0\pmod {2^{k+1}} \quad \text{and} \quad 2^l\binom{2^{k+1}j+2^k+2^ki}{l}\equiv 0\pmod {2^{k+1}}
			\end{align*}
			in the above equation and then extracting the terms involving $q^{2n+1}$, we obtain
			\begin{align*}
				\sum_{n=0}^{\infty}\bar{a}_{2^{k+1}j+2^k, 2^ki}(4n+2)q^n&\equiv\dfrac{f_2^{5(2^ki)}f_4^{3(2^ki)+5(2^{k+1}j+2^k)}}{f_1^{2(3(2^ki)+2(2^{k+1}j+2^k))}f_{8}^{2(2^ki+2^{k+1}j+2^k)}}\pmod {2^{k+1}}\\
				&=\dfrac{f_2^{5(2^ki)}f_4^{3(2^ki)+5(2^{k+1}j+2^k)}}{f_{8}^{2(2^ki+2^{k+1}j+2^k)}}\displaystyle\left(\dfrac{f_8^5}{f_2^5f_{16}^2}+2q\dfrac{f_4^2f_{16}^2}{f_2^5f_8}\right)^{3\cdot2^ki+2^{k+2}j+2^{k+1}}\\
				&=\dfrac{f_4^{3(2^ki)+5(2^{k+1}j+2^k)}f_{8}^{5(3\cdot2^ki+2^{k+2}j+2^{k+1})-2(2^ki+2^{k+1}j+2^k)}}{f_2^{5(2^{k+1}i+2^{k+2}j+2^{k+1})}f_{16}^{2(3\cdot2^ki+2^{k+2}j+2^{k+1})}}\\
				&\quad \displaystyle\left(\sum_{t=0}^{3\cdot2^ki+2^{k+2}j+2^{k+1}}\binom{3\cdot2^ki+2^{k+2}j+2^{k+1}}{t}2^tq^t \dfrac{f_4^{2t}f_{16}^{4t}}{f_8^{6t}}\right).
			\end{align*}
			Again using the fact, for all $t\ge 1$, 
			\begin{align*}
				2^t\binom{3\cdot2^ki+2^{k+2}j+2^{k+1}}{t}\equiv 0\pmod {2^{k+1}}
			\end{align*}
			and then extracting the terms involving $q^{2n+1}$ in the resulting identity, we complete the proof of \eqref{e1.68}.\\
			\indent	Finally, setting $(r,s)=(2^{k+1}j+2^k-1, 2^ki+1)$ in \eqref{e21n}, with similar arguments we obtain \eqref{e1.69}. We, therefore, omit the details involved.
		\end{proof}
		
		\section{Congruences modulo prime $p\ge 3$}\label{s5}
		In this section, we focus our attention on proving Theorem \ref{tmod3} and \ref{tmodp}.
		\begin{proof}[\textbf{Proof of Theorem \ref{tmod3}}]
			From \eqref{gf} and Lemma \ref{cai}, we have
			\begin{align*}
				\sum_{n=0}^{\infty}\bar{a}_{3,6}(n)q^n&=\phi(q)^{6}\prod_{i\ge1}\phi(q^{2^i})^{9\cdot2^{i-1}}\\
				&\equiv\phi(q^3)^{2}\prod_{i\ge1}\phi(q^{3\cdot2^i})^{3\cdot2^{i-1}}\pmod 3,
			\end{align*}
			which is a function of $q^3$. Hence, extracting terms of the form $q^{3n+1}$ and $q^{3n+2}$, we obtain $j,k=0$ case of \eqref{e1.20} and \eqref{e1.21}, respectively.
			Using Lemma \ref{lmodp} with $A(n)=\bar{a}_{3, 6}(n)$, $B(n)=\bar{a}_{3(k-j)+3, 3k+6}(n)$, $(p,\lambda)=(3,1)$, and $C=1$ and $C=2$, we complete the proof of \eqref{e1.20} and \eqref{e1.21}, respectively.\\
			Again, thanks to \eqref{gf} and Lemma \ref{cai}, we have
			\begin{align*}
				\sum_{n=0}^{\infty}\bar{a}_{5,1}(n)q^n&=\phi(q)\prod_{i\ge1}\phi(q^{2^i})^{6\cdot2^{i-1}}\\
				&\equiv\phi(q)\prod_{i\ge1}\phi(q^{3\cdot2^i})^{2^{i}}\pmod 3.
			\end{align*}
			We observe that $3n+2\equiv k^2\pmod 3$ holds if and only if $2\equiv k^2\pmod 3$. Since $2$ is a quadratic non-residue modulo $3$, we obtain $j,k=0$ case of \eqref{e1.22}.
			Upon employing Lemma \ref{lmodp} with $A(n)=\bar{a}_{5,1}(n)$, $B(n)=\bar{a}_{3(k-j)+5, 3k+1}(n)$, $(p,\lambda)=(3,1)$ and $C=1$, we complete the proof of \eqref{e1.22}.
		\end{proof}
		
		\begin{proof}[\textbf{Proof of Theorem \ref{tmodp}}]
			Thanks to \eqref{gf} and Lemma \ref{cai}, we have 
			\begin{align*}
				\sum_{n=0}^{\infty}\bar{a}_{p-1,p+1}(n)q^n&=\phi(q)^{p+1}\prod_{i\ge1}\phi(q^{2^i})^{2p\cdot2^{i-1}}\\
				&\equiv\phi(q)\phi(q^p)\prod_{i\ge1}\phi(q^{p\cdot2^i})^{2^{i}}\pmod p.
			\end{align*}
			We consider, $pn+r\equiv k^2\pmod p$ which is true if and only if $r\equiv k^2\pmod p$. Since $r$ is assumed to be a quadratic non-residue modulo $p$, we obtain $j,k=0$ case of \eqref{e1.23}.
			Using Lemma \ref{lmodp} with $A(n)=\bar{a}_{p-1, p+1}(n)$, $B(n)=\bar{a}_{p(k-j)+p-1, pk+(p+1)}(n)$, $(p,\lambda)=(p,1)$ and $C=r$ we complete the proof of \eqref{e1.23}.\\
			Again thanks to \eqref{gf} and \eqref{phim}, we have
			\begin{align*}
				\sum_{n=0}^{\infty}a_{p-1, p}(n)q^n&=\dfrac{f_2^{p+2}}{f_1^{2p}f_4}\\
				&\equiv \dfrac{f_{2p}}{f_p^2}\dfrac{f_2^2}{f_4}\pmod p\\
				&=\dfrac{f_{2p}}{f_p^2}\phi(-q^2)\\
				&=\dfrac{f_{2p}}{f_p^2}\sum_{k=-\infty}^{\infty}(-q^2)^{k^2}.
			\end{align*}
			We consider, $pn+r\equiv 2k^2\pmod p$ which is true if and only if $r\equiv 2k^2\pmod p$. Since $2^{-1}r$ is assumed to be a quadratic non-residue modulo $p$, we obtain $j,k=0$ case of \eqref{e1.24}. Using Lemma \ref{lmodp} with $A(n)=\bar{a}_{p-1, p}(n)$, $B(n)=\bar{a}_{p(k-j)+p-1, pk+p}(n)$, $(p,\lambda)=(p,1)$ and $C=r$, we complete the proof of \eqref{e1.24}.
		\end{proof}
		\section{Conclusion}
		We encourage the interested reader to further investigate arithmetic properties of $\bar{a}_{r,s}(n)$. In particular, we would be very much pleased to see an elementary proof of congruences in Conjecture \ref{c6.1}.
		\begin{conjecture}\label{c6.1}
			For $k\ge1$ and $j, i\ge 0$, we have
			\begin{align}
				\bar{a}_{2^{k+1}j+2^k-1,2^ki+1}(3n+2)&\equiv 0 \pmod {2^{k+1}}\\
				\bar{a}_{2^{k+1}j+2^k-1,2^ki+1}(9n+3)&\equiv 0 \pmod {2^{k+2}}\\
				\bar{a}_{2^{k+1}j+2^k-1,2^ki+1}(9n+6)&\equiv 0 \pmod {2^{k+2}}.
			\end{align}
		\end{conjecture}

\noindent\textbf{Acknowledgment.} The authors would like to thank Professor James A. Sellers for his suggestions and constructive comments which has substantially improved the paper.

	\bigskip
	\bigskip
	
	\noindent
Department of Mathematics\\
Ramanujan School of Mathematical Sciences\\
Pondicherry University\\
Puducherry- 605 014, India.\\

\noindent Email: \texttt{tthejithamp@pondiuni.ac.in}

\noindent	Email: \texttt{dr.fathima.sn@pondiuni.ac.in} (\Letter)

\end{document}